\documentclass[ a4paper,
                reqno, 
                twocolumn
                ]{article}
                
\usepackage{amsmath, amssymb, amscd, amsthm, mathtools, comment}

\usepackage{extarrows, hyperref}
\usepackage{tikz}
\usetikzlibrary{matrix,arrows,decorations.markings,decorations.pathreplacing,petri,topaths}

\usepackage{sectsty}
\sectionfont{\large}

\usepackage{abstract}

\newcommand{\N}{\ensuremath{\mathbb{N}}}
\newcommand{\Z}{\ensuremath{\mathbb{Z}}}
\newcommand{\R}{\ensuremath{\mathbb{R}}}
\newcommand{\Y}{\ensuremath{\mathbb{Y}}}
\renewcommand{\S}{\ensuremath{\mathbb{S}}}
\renewcommand{\rho}{\varrho}
\renewcommand{\epsilon}{\varepsilon}
\DeclareMathOperator{\im}{im}

\newtheorem{thm}{Theorem}
\newtheorem{lem}[thm]{Lemma}

\newtheorem{defn}[thm]{Definition}
\newtheorem{prop}[thm]{Proposition}

\newcommand{\doi}[1]{\href{http://dx.doi.org/#1}{DOI:{#1}}}

\newcommand{\cT}[7]{
\draw (#1,#2) rectangle +(1,1);
\node at (#1+0.5, #2+0.5)
{${\scriptstyle #3}\negthinspace\overset{#4}{\underset{#6}{#7}}\negthinspace{\scriptstyle
#5}$};
}

\begin{document}

\title{
\textbf{Quotient cohomology of certain 1- and 2-dimensional \\ substitution tiling spaces}
}

\author{
Enrico Paolo Bugarin and Franz G\"ahler
}

\date{
{\small Fakult\"at f\"ur Mathematik, Universit\"at Bielefeld \\ 
Universit\"atsstra{\ss}e 25, D-33615 Bielefeld, Germany}
}

\twocolumn[
\begin{@twocolumnfalse}

\maketitle

\begin{abstract} 

\noindent The quotient cohomology of tiling spaces is a topological invariant that relates a tiling space to one of
its factors, viewed as topological dynamical systems. In particular, it is a relative version of the tiling
cohomology that distinguishes factors of tiling spaces. In this work, the quotient cohomologies within certain
families of substitution tiling spaces in 1 and 2 dimensions are determined. Specifically, the quotient cohomologies
for the family of the generalised Thue-Morse sequences and generalised chair tilings are presented.
     
\smallskip


\bigskip

\end{abstract}

\end{@twocolumnfalse}
]

{
  \renewcommand{\thefootnote} 
    {\fnsymbol{footnote}}
  \footnotetext[1]{email: \texttt{\{pbugarin, gaehler\}@math.uni-bielefeld.de}}
}


\section{Introduction} 

Characterising tiling spaces through topological invariance can provide a systematic way of distinguishing tiling
spaces. In particular, tiling spaces with non-isomorphic (\v{C}ech) cohomology groups are necessarily inequivalent.
The converse though is not true in general, as can be seen in the case of the classical Thue-Morse and
period-doubling sequences, which form two inequivalent tiling spaces with isomorphic cohomology groups. However,
for tiling spaces related by a factor map (regarding the spaces as topological dynamical systems), a relative 
version of the tiling cohomology can be used to tell the spaces apart. Barge and Sadun \cite{bs11} introduced the
concept of \emph{quotient cohomology}, which is a topological invariant that distinguishes factors of tiling spaces.
In this paper, we present the quotient cohomologies for the families of generalised Thue-Morse sequences and
generalised chair tilings. In recomputing the quotient cohomologies of the generalised chair tilings, we find
some discrepancies with \cite{bs11}, which we address below.

\subsection{Preliminaries} 

A primitive (tiling) \emph{substitution rule} is a recipe of constructing tilings of $\R^d$ using only a finite set
of tile types, called prototiles. The rule prescribes on how each prototile is scaled linearly (by a fixed inflation
factor) and then is subdivided into a collection of smaller tiles called a {supertile} (of order 1), all of 
which are translate copies of some prototiles. Applying the substitution rule to any (proto)tile $k$ times produces a
supertile of order $k$, and the limit of the process produces a tiling of $\R^d$, called a \emph{substitution tiling}.

A tiling $T$ of $\R^d$ arising from a substitution rule $\omega$ defines a \emph{substitution tiling space} as
its \emph{hull}. The hull of $T$, denoted by $\Omega_T$, is the closure of its translation orbit under a metric,
where two tilings are ``$\epsilon$-close'' if they agree on a ball of radius $\epsilon^{-1}$ around the origin, 
after a translation of at most $\epsilon$ in any direction. Two tilings $T$ and $T'$ arising from the same
substitution rule $\omega$ define the same hull, and so the tiling space is instead associated with
a substitution rule rather than a particular substitution tiling, i.e., $\Omega_\omega := \Omega_T = \Omega_{T'}$.

A substitution tiling space can be represented as an inverse limit of simpler spaces called approximants, relative 
to a continuous bonding map induced by the substitution rule $\omega$, see \cite{ap98, sa08} for more details. When
the prototiles in $\Omega_\omega$ are homeomorphic to a disk in $\R^d$, then the approximants are $d$-dimensional
CW complexes, which are also known as the Anderson-Putnam complexes. The cohomology of $\Omega_\omega$ is then
computed as the direct limit of the cohomologies of the approximants under the homomorphism induced by $\omega$.

Substitution tiling spaces are minimal dynamical systems, and factor maps between these spaces are surjective
and generally not injective. (And so the relative cohomology is not immediately available.) A factor map 
$f: \Omega_X \rightarrow \Omega_Y$ induces a quotient map (denoted by the same symbol) on the level of approximants
that is also surjective and generally not injective. This motivates the following definition of the quotient
cohomology \cite{bs11}.

\begin{defn}\label{def:quotcoh} 
Let $f: X \rightarrow Y$ be a quotient map such that $f^*$ is injective on cochains. Also, let 
$C^k_Q(X,Y) := C^k(X)/f^*(C^k(Y))$ be the quotient cochain groups and take 
$\delta_k : C^k_Q(X,Y) \rightarrow C_Q^{k+1}(X,Y)$ to be the usual coboundary operator. The ($k$th)
\emph{quotient cohomology} is defined as $H^k_Q(X,Y) := \ker \delta_k / \im\delta_{k-1}.$
\end{defn}

\noindent By the snake lemma, the short exact sequence of cochain complexes 
$$0 \longrightarrow C^k(Y) \xlongrightarrow{f^*} C^k(X) \longrightarrow C^k_Q(X, Y) \longrightarrow 0$$
induces a long exact sequence
\begin{equation}\label{eq:longseq}
\begin{gathered}
\cdots \longrightarrow H^{k-1}_Q(X,Y) \longrightarrow H^k(Y) \xlongrightarrow{f^*_k} \\ 
H^k(X) \longrightarrow H^k_Q(X,Y)\longrightarrow \cdots
\end{gathered} 
\end{equation}
that relates the cohomologies of $X$ and $Y$ to $H^*_Q(X,Y)$.

\begin{lem} \label{lem:quotcohom}
Let $f: X \rightarrow Y$ be a quotient map, whose pullback $f^*$ is injective on the cochains. If $H^{n+1}(Y) = 0$,
then $H^n_Q(X,Y) = H^n(X)/f^*_n(H^n(Y))$. For $X$ and $Y$ being approximant spaces for substitution tiling spaces,
$H^0_Q(X,Y) = 0$ if and only if $f^*_1: H^1(Y) \rightarrow H^1(X)$ is injective.
\end{lem}

\begin{proof}
If $X$ and $Y$ are approximant spaces for substitution tiling spaces, then 
$H^0(X) = \Z = \langle\sum x'_i\rangle = \langle\sum f^*(y'_j)\rangle =f^*_0(H^0(Y)) = f^*_0(\Z)$, where the 
$x'_i$'s and $y'_j$'s are the duals to the $0$-cells in $X$ and $Y$ respectively. Thus $f^*_0$ is surjective, and so
is an isomorphism. Further, the map from $H^0(X)$ to $H^0_Q(X,Y)$ in \eqref{eq:longseq} must be a zero map and so 
the map from $H^0_Q(X,Y)$ to $H^1(Y)$ must be injective. If $f^*_1$ is injective, then the map from $H^0_Q(X,Y)$ to
$H^1(Y)$ must be a zero map as well, which is already shown to be injective, forcing $H^0_Q(X,Y) = 0$. Conversely, 
if $H^0_Q(X,Y)=0$, then $f_1^*$ must be injective since the sequence in \eqref{eq:longseq} is exact. Meanwhile,
$H^{n+1}(Y)=0$ implies that the map $H^n(X)$ to $H^n_Q(X,Y)$ is surjective, and so it follows that 
$H^n_Q(X,Y) = H^n(X)/f^*_n(H^n(Y))$.
\end{proof}

All substitution tiling spaces considered in this work yield $H^0_Q = 0$.

\section{Generalised Thue-Morse sequences} 

For any $k, \ell \in \N$, the substitution rules
\begin{equation}\label{eq:TMpdsubsitutionrules}
\begin{array}{ll}
\varrho^{TM}_{k,\ell}&:=\left\{\begin{array}{rcl}    1  & \longmapsto & 1^k \, \bar{1}^\ell \\
                                                \bar{1} & \longmapsto & \bar{1}^k \, 1^\ell
                              \end{array}\right. 
\\ \\[-7pt]
\varrho^{pd}_{k,\ell}&:=\left\{\begin{array}{rcl}    a  & \longmapsto & b^{k-1} \, a \, b^{\ell-1} \, b \\ 
                                                     b  & \longmapsto & b^{k-1} \, a \, b^{\ell-1} \, a
                              \end{array}\right.
\end{array}
\end{equation}
define the hulls $\Y^{TM}_{k,\ell}$ (generalised Thue-Morse) and $\Y^{pd}_{k,\ell}$ (generalised period doubling)
respectively. The case $k=\ell=1$ yields the classic Thue-Morse and period doubling sequences. For a detailed
exposition on the spectral and topological properties of the generalised Thue-Morse sequences, we refer the readers 
to \cite{bgg12}.

Each letter in \eqref{eq:TMpdsubsitutionrules} becomes a tile in $\R$ by assigning the same constant length to any 
one of them, so that a letter becomes a closed interval in $\R$. In turn, every bi-infinite sequence arising from
\eqref{eq:TMpdsubsitutionrules} tiles $\R$ in an obvious way. We also consider the 1-dimensional solenoid
$\S_{k+\ell}$, which can be viewed as the inverse limit of 1-dimensional tori under the bonding maps that uniformly
wrap a torus $k+\ell$ times around its predecessor. The solenoid $\S_{k+\ell}$ may be realised as the inverse limit 
of the substitution $s \longmapsto s^{k+\ell}$, though strictly speaking, the solenoid is not a tiling space because
tilings generated by this substitution are all periodic. The solenoid has as additional information the partitioning
of these tilings into supertiles of all orders. For convenience, we may nevertheless use the term `tiling space' 
even for solenoids in the following.

The three spaces are related via the factor maps $\phi$ and $\psi$, namely
\begin{equation*}\label{eq:gtmfactormaps}
\Y^{TM}_{k,\ell} \xlongrightarrow{\phi} \Y^{pd}_{k,\ell} \xlongrightarrow{\psi} \S^{}_{k+\ell},
\end{equation*} 
where $\phi$ is a sliding block map that identifies $\{1\bar{1}, \bar{1}1\}$ with $a$, and $\{11, \bar{1}\bar{1}\}$
with $b$; while $\psi$ simply identifies $a$ and $b$ with $s$. Note that $\phi$ is uniformly 2-to-1, whereas $\psi$ 
is a surjection that is 1-to-1 almost everywhere. Each of these factor maps induces a pullback on their respective
cohomologies given by
\begin{equation*} 
H^*(\S^{}_{k+\ell}) \xlongrightarrow{\psi^*} H^*(\Y^{pd}_{k,\ell}) \xlongrightarrow{\phi^*} H^*(\Y^{TM}_{k,\ell}).
\end{equation*}
As computed in \cite{bgg12}, the cohomologies of the three tiling spaces are: $H^0 = \Z$ and
\begin{gather*}
H^1(\S^{}_{k+\ell}) = \Z[\tfrac{1}{k+\ell}], \quad H^1(\Y^{pd}_{k,\ell}) = \Z[\tfrac{1}{k+\ell}] \oplus \Z, \\
H^1(\Y^{TM}_{k,\ell}) = \Z[\tfrac{1}{k+\ell}] \oplus \Z \oplus \Z[\tfrac{1}{|k-\ell|}],
\end{gather*}
where $\Z[\tfrac{1}{0}] := 0$ by an abuse of notation.

\begin{thm} For any $k,\ell \in \N$, $H^0_Q = 0$. Further,
\begin{align*}
H^1_Q(\Y_{k,\ell}^{pd},\S^{}_{k+\ell}) &= \Z, \\
H^1_{Q}(\Y_{k,\ell}^{TM},\S^{}_{k+\ell}) &= \Z [\tfrac{1}{|k-\ell|}] \oplus \Z, \\
H^1_{Q}(\Y_{k,\ell}^{TM},\Y_{k,\ell}^{pd}) &= \Z_2 \oplus \Z [\tfrac{1}{|k-\ell|}],
\end{align*} 
where $\Z[\tfrac{1}{0}] := 0$ by an abuse of notation.
\end{thm}
\begin{proof}
The map $\psi^*$ embeds $H^1(\S_{k+\ell}) = \Z[\frac{1}{k+\ell}]$ isomorphically onto the same summand in
$H^1(\Y^{pd}_{k,\ell})$, which the map $\phi^*$ also embeds isomorphically onto the same summand in
$H^1(\Y^{TM}_{k,\ell})$. Thus, $\varphi^* := \phi^*\circ\psi^*$ also embeds 
$H^1(\S_{k+\ell}) = \Z[\frac{1}{k+\ell}]$ isomorphically onto $\Z[\frac{1}{k+\ell}]$ in $H^1(\Y^{TM}_{k,\ell})$.
Furthermore, $\phi^*$ maps the summand $\Z$ in $H^1(\Y^{pd}_{k,\ell})$ onto $2\Z$ in $H^1(\Y^{TM}_{k,\ell})$ (see
\cite{bgg12}). Thus, the maps $\psi^*_1$, $\phi^*_1$ and $\varphi^*_1$ are all injective maps, and so $H^0_Q = 0$ 
for all spaces using Lemma \ref{lem:quotcohom}. By the same lemma, we get 
$H^1_{Q}(\Y_{k,\ell}^{TM},\Y_{k,\ell}^{pd}) = H^1(\Y^{TM}_{k,\ell})/\phi^*_1(H^1(\Y_{k,\ell}^{pd})) 
= \Z_2 \oplus \Z [\tfrac{1}{|k-\ell|}]$. The rest of the results follow similarly.
\end{proof}%

The three non-trivial quotient cohomologies are related via the short exact sequence
\begin{gather*}
0 \longrightarrow H^1_Q(\Y^{pd}_{k,\ell}, \S^{}_{k+\ell}) \xlongrightarrow{\times 2} H^1_Q(\Y^{TM}_{k,\ell},
\S^{}_{k+\ell}) \longrightarrow \\ H^1_Q(\Y^{TM}_{k,\ell}, \Y^{pd}_{k,\ell}) \longrightarrow 0.
\end{gather*}

\section{Generalised chair tilings} 

The hull of the classic chair tiling, c.f. \cite{rob99}, (also called triomino tiling \cite{ap98}) 
defined by the substitution rule

$$
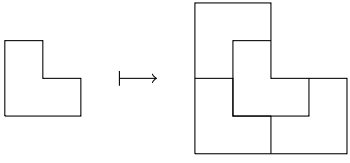
\begin{tikzpicture}[scale=1] 
  \draw (0,0.5) -| (1,1) -| (0.5,1.5) -| (0,0.5);
  \draw [|->] (1.5,1) -- +(0.5, 0);
  \draw (2.5,0) -| (4.5,1) -| (3.5,2) -| (2.5,0);
  \draw (3.5,1.5) -| (3,0.5) -| (4,1);
  \draw (2.5,1) -| (3,0.5) -| (3.5,0);
  \end{tikzpicture}$$ 
belongs to a family of substitution tiling spaces called the \emph{generalised chair tilings}, which Barge and Sadun
introduced and analysed in \cite{bs11}. Using decorated square tiles as prototiles, the most intricate of them
defines the tiling space $\Omega_{X,+}$ through the substitution rule

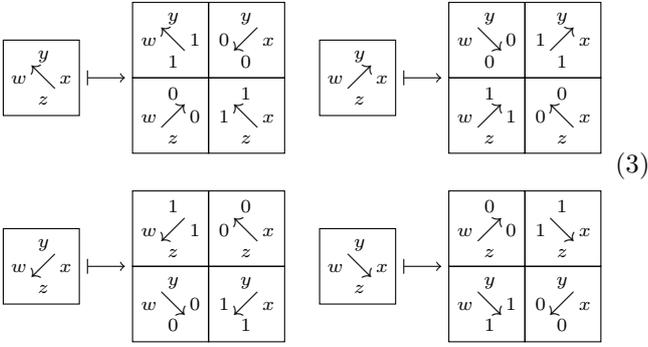
\begin{equation}\label{eq:chairsubstirule}
\begin{tikzpicture}[baseline=(current bounding box.center)]
  \def\x{0}
  \def\y{0}
  \def\xx{\x+4.16}
  \def\yy{\y-2.5}
  
  \cT{\x}{\y}{w}{y}{x}{z}{\nwarrow}
  \draw [|->] (\x+1.1,\y+0.5) -- +(0.5, 0);
  \cT{\x+1.7}{\y+0.5}{w}{y}{1}{1}{\nwarrow}
  \cT{\x+2.7}{\y+0.5}{0}{y}{x}{0}{\swarrow}
  \cT{\x+2.7}{\y-0.5}{1}{1}{x}{z}{\nwarrow}
  \cT{\x+1.7}{\y-0.5}{w}{0}{0}{z}{\nearrow}
  
  \cT{\x}{\yy}{w}{y}{x}{z}{\swarrow}
  \draw [|->] (\x+1.1,\yy+0.5) -- +(0.5, 0);
  \cT{\x+1.7}{\yy-0.5}{w}{y}{0}{0}{\searrow}
  \cT{\x+2.7}{\yy-0.5}{1}{y}{x}{1}{\swarrow}
  \cT{\x+2.7}{\yy+0.5}{0}{0}{x}{z}{\nwarrow}
  \cT{\x+1.7}{\yy+0.5}{w}{1}{1}{z}{\swarrow}
  
  \cT{\xx}{\y}{w}{y}{x}{z}{\nearrow}
  \draw [|->] (\xx+1.1,\y+0.5) -- +(0.5, 0);
  \cT{\xx+1.7}{\y+0.5}{w}{y}{0}{0}{\searrow}
  \cT{\xx+2.7}{\y+0.5}{1}{y}{x}{1}{\nearrow}
  \cT{\xx+2.7}{\y-0.5}{0}{0}{x}{z}{\nwarrow}
  \cT{\xx+1.7}{\y-0.5}{w}{1}{1}{z}{\nearrow}
  
  \cT{\xx}{\yy}{w}{y}{x}{z}{\searrow}
  \draw [|->] (\xx+1.1,\yy+0.5) -- +(0.5, 0);
  \cT{\xx+1.7}{\yy-0.5}{w}{y}{1}{1}{\searrow}
  \cT{\xx+2.7}{\yy-0.5}{0}{y}{x}{0}{\swarrow}
  \cT{\xx+2.7}{\yy+0.5}{1}{1}{x}{z}{\searrow}
  \cT{\xx+1.7}{\yy+0.5}{w}{0}{0}{z}{\nearrow}
  \end{tikzpicture}
\end{equation} 
where $w,x,y,z \in \{0,1\}$ and with the two labels adjacent to the head of an arrow being the same.   
  
Factors of $\Omega_{X,+}$ can be defined by removing and/or identifying certain decorations on the square tiles to
which the general substitution rule \eqref{eq:chairsubstirule} applies. Tiling spaces $\Omega_{a,b}$, with $a \in
\{X, /, 0\}$ and $b \in \{+, -, 0\}$, are defined through the following:
\begin{center}
\begin{tabular}{|c|c|p{6.5cm}|}
\hline
\multicolumn{2}{|c|}{Index} & Description \\ \hline

$a$ & $X$ & The four arrows on the square tiles remain. \\
    & $/$ & Only the arrows pointing northeast or southwest remain, i.e., arrowheads pointing to other directions 
            are identified. \\
    & $0$ & All arrows are identified/removed. \\ \hline
$b$ & $+$ & All four labels remain. \\
    & $-$ & Only the labels to the left or to the right remain, i.e., the top and bottom labels are identified. \\
    & $0$ & All labels are identified/removed. \\ \hline
\end{tabular}
\end{center}

\smallskip

\noindent In particular, $\Omega_{X,0}$ is (equivalent to) the chair tiling space and $\Omega_{0,0}$ is the
2-dimensional dyadic solenoid $\S_2 \times \S_2$. The scheme above yields nine tiling spaces that are related as
follows:

\begin{equation}\label{eq:9models}\begin{CD}
\Omega_{X,+} @>A>> \Omega_{/,+} @>A>> \Omega_{0,+} \\
                                 @VVBV @VVBV @VVBV \\
\Omega_{X,-} @>A>> \Omega_{/,-} @>A>> \Omega_{0,-} \\
                             @VV{A}V @VV{A}V @VVCV \\
\Omega_{X,0} @>A>> \Omega_{/,0} @>C>> \Omega_{0,0} \\
\end{CD}\end{equation}

\noindent Barge and Sadun beautifully computed the quotient cohomology between adjacent tiling spaces appearing in
\eqref{eq:9models}, using a framework discussed in \cite{bs11}. The quotient cohomologies are given by:
\begin{equation} \label{eq:BSadjacentcohomology}
\begin{aligned}
A &: & H^0_Q &= 0, & H^1_Q & = \Z, & H^2_Q &= \Z[\tfrac{1}{2}], \\
B &: & H^0_Q &= 0, & H^1_Q & = \Z, & H^2_Q &= \Z[\tfrac{1}{2}] \oplus \Z, \\
C &: & H^0_Q &= 0, & H^1_Q & = 0,  & H^2_Q &= \Z[\tfrac{1}{2}] \oplus \Z. 
\end{aligned}
\end{equation}

Tracing a path in \eqref{eq:9models} pertains to a factor map from one tiling space (starting point) to another 
tiling space (ending point). As such, paths having identical starting and ending points pertain to equivalent factor
maps. In this sense, we say that the diagram commutes. We generalise the results in \eqref{eq:BSadjacentcohomology} 
by giving the quotient cohomology between tiling spaces in \eqref{eq:9models}, depending on the factor map between
them as obtained by tracing an arbitrary path. We formalise this as the following theorem, whose calculation is
straightforward.

\begin{thm}\label{thm:generalisedchair}
The quotient cohomologies between adjacent tiling spaces in \eqref{eq:9models} are given in
\eqref{eq:BSadjacentcohomology}; for the remaining pairs of tiling spaces, we have:
\begin{align*}
AA   &: & H^0_Q &= 0, & H^1_Q & = \Z^2,  & H^2_Q &= \Z[\tfrac{1}{2}]^2, \\
AB   &: & H^0_Q &= 0, & H^1_Q & = \Z^2,  & H^2_Q &= \Z[\tfrac{1}{2}]^2 \oplus \Z, \\
AAB  &: & H^0_Q &= 0, & H^1_Q & = \Z^3,  & H^2_Q &= \Z[\tfrac{1}{2}]^3 \oplus \Z, \\
BC   &: & H^0_Q &= 0, & H^1_Q & = \Z,    & H^2_Q &= \Z[\tfrac{1}{2}]^2 \oplus \Z^2, \\
AC   &: & H^0_Q &= 0, & H^1_Q & = 0,     & H^2_Q &= \Z_3 \oplus \Z[\tfrac{1}{2}]^2, \\
AAC  &: & H^0_Q &= 0, & H^1_Q & = \Z,    & H^2_Q &= \Z_3 \oplus \Z[\tfrac{1}{2}]^3, \\
BAC  &: & H^0_Q &= 0, & H^1_Q & = \Z,    & H^2_Q &= \Z_3 \oplus \Z[\tfrac{1}{2}]^3 \oplus \Z, \\
ABAC &: & H^0_Q &= 0, & H^1_Q & = \Z^2,  & H^2_Q &= \Z_3 \oplus \Z[\tfrac{1}{2}]^4 \oplus \Z.
\end{align*} \qed
\end{thm}

\noindent Note that the quotient cohomology depends only on the type of path, and not necessarily on particular
tiling spaces. Also, the quotient cohomology groups sum up whenever factor maps are composed. The only exception is
when composing $A$ and $C$ which produces the torsion component $\Z_3$. For the rest of the compositions, the
operation is associative and commutative.

The following propositions already appear in \cite{bs11} as Theorems 6 and 7, although with some errors. The 
absolute and quotient cohomologies have been recalculated and the corrected results appear in the following. The
particular corrections are boxed for easier identification. Proposition \ref{prop:thm7} may also be read off of
Theorem \ref{thm:generalisedchair}.

\begin{prop}[{cf. \cite[Theorem 6]{bs11}}] \label{prop:thm6}
The absolute cohomologies of the nine tiling spaces in \eqref{eq:9models} are given as follows. All spaces have 
$H^0 = \Z$. The first cohomology is given by
$$\begin{CD}
\Z[\frac{1}{2}]^2 \oplus \Z^2 @<A^*<< \boxed{\Z[\tfrac{1}{2}]^2 \oplus \Z} @<A^*<< \Z[\frac{1}{2}]^2 \oplus \Z \\
                                                                                       @AAB^*A @AAB^*A @AAB^*A \\
\Z[\frac{1}{2}]^2 \oplus \Z   @<A^*<< \boxed{\Z[\tfrac{1}{2}]^2}           @<A^*<< \Z[\frac{1}{2}]^2 \\
                                                                         @AA{A^*}A @AA{A^*}A @AAC^*A \\
\Z[\frac{1}{2}]^2             @<A^*<< \Z[\frac{1}{2}]^2                    @<C^*<< \Z[\frac{1}{2}]^2 \\
\end{CD}$$ %
The second cohomology is given by

$$\begin{CD}
\begin{gathered}\tfrac{1}{3}\Z[\tfrac{1}{4}] \oplus \Z[\tfrac{1}{2}]^4 \\[-8pt] \oplus \, \Z \end{gathered} 
@<A^*<< 
\boxed{\begin{gathered}\tfrac{1}{3}\Z[\tfrac{1}{4}] \oplus \Z[\tfrac{1}{2}]^3 \\[-8pt] \oplus \, \Z \end{gathered}} 
@<A^*<< 
\begin{gathered}\Z[\tfrac{1}{4}] \oplus \Z[\tfrac{1}{2}]^2 \\[-8pt] \oplus \, \Z^2 \end{gathered}
\\
@AAB^*A @AAB^*A @AAB^*A 
\\
\frac{1}{3}\Z[\frac{1}{4}] \oplus \Z[\frac{1}{2}]^3 
@<A^*<< 
\boxed{\tfrac{1}{3}\Z[\tfrac{1}{4}]\oplus \Z[\tfrac{1}{2}]^2} 
@<A^*<<
\begin{gathered}\Z[\tfrac{1}{4}] \oplus \Z[\tfrac{1}{2}] \\[-8pt] \oplus \, \Z \end{gathered}
\\
@AA{A^*}A @AA{A^*}A @AAC^*A 
\\
\frac{1}{3}\Z[\frac{1}{4}]\oplus\Z[\frac{1}{2}]^2 
@<A^*<< 
\Z[\frac{1}{4}]\oplus\Z[\frac{1}{2}]\oplus\Z 
@<C^*<<
\Z[\frac{1}{4}] 
\\
\end{CD}$$ \qed
\end{prop}

\begin{prop}[{cf. \cite[Theorem 7]{bs11}}] \label{prop:thm7}
The quotient cohomologies of the nine tiling spaces in \eqref{eq:9models}, relative to the solenoid $\Omega_{0,0}$,
are given as follows. For all spaces, $H^0_Q=0$. The first quotient cohomology is given by
$$\begin{CD}
\Z^2 @<A^*<< \boxed{\Z} @<A^*<< \Z \\
@AAB^*A @AAB^*A @AAB^*A \\
\Z @<A^*<< \boxed{0} @<A^*<< 0 \\
@AA{A^*}A @AA{A^*}A @AAC^*A \\
0 @<A^*<< 0 @<C^*<< 0 \\
\end{CD}$$ %
The second quotient cohomology is given by

$$\begin{CD}
\Z_3 \oplus \Z[\frac{1}{2}]^4 \oplus \Z @<A^*<< \boxed{\Z_3 \oplus
\Z[\tfrac{1}{2}]^3 \oplus\Z} @<A^*<< \Z[\frac{1}{2}]^2 \oplus \Z^2 \\
@AAB^*A @AAB^*A @AAB^*A \\
\Z_3 \oplus \Z[\frac{1}{2}]^3 @<A^*<< \boxed{\Z_3 \oplus \Z[\tfrac{1}{2}]^2} @<A^*<<
 \Z[\frac{1}{2}] \oplus \Z \\
@AA{A^*}A @AA{A^*}A @AAC^*A \\
\Z_3\oplus\Z[\frac{1}{2}]^2 @<A^*<< \Z[\frac{1}{2}]\oplus\Z @<C^*<< 0 \\
\end{CD}$$ \qed
\end{prop}

\section{Discussion and conclusion} 

Determining the quotient cohomologies between 1-dimensional substitution tiling spaces is rather straightforward
because of Lemma \ref{lem:quotcohom}. In particular, it suffices to know $f^*_1$ to be able to compute both $H^0_Q$
and $H^1_Q$.

In higher dimensions, a first challenge is in the enumeration of inequivalent factors and the factor maps between them
before one can study their quotient cohomologies. In the case of the generalised chair tilings, there are two
more substitution tiling spaces between $\Omega_{X,-}$ and $\Omega_{0,0}$, which are inequivalent to any of those
already enumerated, but are impossible to obtain through the identification rules considered earlier.

A similar analysis as above has also been carried out for the Squiral \cite{bgg13} and Chacon \cite{fra08}
substitution tiling spaces, which are both 2-dimensional.

\section*{Acknowledgement}

This work is supported by the German Research Foundation (DFG) via the Collaborative Research Centre (SFB 701) at
Bielefeld University.

\end{document}